\documentclass[reqno,12pt,a4paper]{amsart}
\usepackage{pgfplots}
\usepackage[english]{babel}
 \pdfoutput=1
\usepackage{amsmath,amsthm,amssymb,amsfonts, mathdots}
\usepackage{scrtime, cancel}
\usepackage{enumitem, color}

\usepackage{mathtools}
\usepackage{ifpdf}
\ifpdf
\usepackage[backref]{hyperref}
\else
\usepackage[hypertex]{hyperref}
\fi
\usepackage{pdfsync,verbatim}
\usepackage{color}
\mathtoolsset{showonlyrefs}

\numberwithin{equation}{section}   %%numera le equazioni sezione per sezione

\allowdisplaybreaks
\newtheorem{theorem}{Theorem}[section]
\newtheorem{lemma}[theorem]{Lemma}
\newtheorem{proposition}[theorem]{Proposition}

\theoremstyle{definition}

\usepackage{graphicx}

\makeatletter
\newcommand*\bigcdot{\mathpalette\bigcdot@{.5}}
\newcommand*\bigcdot@[2]{\mathbin{\vcenter{\hbox{\scalebox{#2}{$\m@th#1\bullet$}}}}}
\makeatother

\newcommand{\R}{\mathbb{R}}

\newcount\dotcnt\newdimen\deltay
\def\Ddot#1#2(#3,#4,#5,#6){\deltay=#6\setbox1=\hbox to0pt{\smash{\dotcnt=1
\kern#3\loop\raise\dotcnt\deltay\hbox to0pt{\hss#2}\kern#5\ifnum\dotcnt<#1
\advance\dotcnt 1\repeat}\hss}\setbox2=\vtop{\box1}\ht2=#4\box2}

\pgfplotsset{compat = newest}

\makeatletter
\def\author@andify{%
  \nxandlist {\unskip ,\penalty-1 \space\ignorespaces}%
    {\unskip {} \@@and~}%
    {\unskip \penalty-2 \space \@@and~}%
}
\makeatother

\title[Variation for Ornstein--Uhlenbeck]{
On the variation operator for the Ornstein--Uhlenbeck  semigroup
in dimension one}

\author{Valentina Casarino}
\address{DTG, Universit\`a degli Studi di Padova\\ Stradella san Nicola 3 \\I-36100 Vicenza \\ Italy}
\email{valentina.casarino@unipd.it}
\author{Paolo Ciatti}
\address{Dipartimento di Matematica "Tullio Levi Civita", Universit\`a degli Studi di Padova\\Via Trieste, 63, 35131 Padova,  \\ Italy}
\email{paolo.ciatti@unipd.it}
\author{Peter Sj\"ogren}
\address{Mathematical Sciences,  University of Gothenburg and  Mathematical Sciences,
Chalmers University of Technology  \\ SE - 412 96 G\"oteborg, Sweden}
\email{peters@chalmers.se}


%semigruppi;

\thanks{
The first and second authors are members of the Gruppo Nazionale per l'Analisi Matematica, la Probabilità e le loro Applicazioni (GNAMPA)
of the Istituto Nazionale di Alta Matematica (INdAM) and were partially supported by GNAMPA (Project 2022
``Temi di Analisi Armonica Subellittica"). 
This research was carried out while the third author was GNAMPA Professore Visitatore at the University of Padova,  Italy. The third author also profited from a grant from the Adlerbert Research Foundation in Sweden.}

\date{\today}

\begin{document}

\begin{abstract}
Consider the variation seminorm of the Ornstein--Uhlenbeck semigroup $H_t$ in dimension one, taken with respect to $t$.  
We show that this seminorm defines an operator of weak type $(1,1)$ for the relevant Gaussian measure. The analogous $L^p$ estimates for $1<p<\infty$ were already known.
\end{abstract}

\maketitle

\section{Introduction}

In order to measure the fluctuations of  a  family of linear bounded operators $A_t:L^p(X)\to L^p(X)$, where $t > 0$ and  $X$ is a measure space,
%of a strongly continuous semigroup $(T_t)_{t\ge 0} of linear  bounded operators,
it may be useful
to consider quantities involving many differences $T_t f(\cdot)-T_sf(\cdot)$, with $s,t>0$ and $f\in L^p(X)$.
Among  these quantities, variation and oscillation seminorms are probably  the best known.   %%Altri operatori di questo tipo sono, per esempio, i $\lambda$-jumps.
The corresponding variational
inequalities, stating that the $L^p$-norm of the variation or the oscillation of $(A_t f)_{t > 0}$ is uniformly bounded by the $L^p$-norm of $f$,
     have attracted increasing interest in the last fifty years.

In fact, in 1976
 D. Lépingle
 proved a first variational inequality for a family of bounded martingales   \cite{Lepingle}, also providing   a weak type $(1, 1)$ variant
 (see also \cite{MSZK} for  extensions and
 a different proof).
Then V. F. Gaposhkin in \cite{Gaposhkin1, Gaposhkin2}
considered oscillational inequalities for  standard ergodic averages.
%%A variation inequality    was  first proved.  by D. Lépingle  in 1978 for a family of bounded martingales  \cite{Lepingle}. In that paper, Lépingle proved a  a weak type $(1, 1)$ variant.
Some years later, in 1989,
J. Bourgain proved  the pointwise convergence of ergodic averages along polynomial orbits
by replacing the classical estimates for  the Hardy--Littlewood maximal function  by  variational seminorm bounds \cite{Bourgain}.
%%{\Red{His results were then extended to the higher dimensional case and to weak type bounds in \cite{Jones1-higher, Jones} (da cancellare).}}
Further results may be found 
   in \cite{Jones1-higher, Jones}.
   
%%employed the tool of  variational inequalities % involving  variation and oscillation operators
%%to prove  the pointwise convergence of ergodic averages along polynomial orbits  \cite{Bourgain}.

After the seminal work by Bourgain, %and their higher-dimensional and weak type,
in light of the applications to  pointwise convergence phenomena,  the study of  variational inequalities
spread in many different contexts (for an updated survey, especially  from the point of view of oscillation estimates, we refer to \cite{MSW}).
The recent paper \cite{Mirek_boot}
deals with jump inequalities, seen as endpoint results for variation inequalities. 
Focusing on the field of harmonic analysis, we  recall here the cases of  the Hilbert transform \cite{Campbell}, Fejér and Poisson kernels \cite{Jones4},
  families of truncations of Gaussian Riesz transforms \cite{Harboure} and
 the heat and Poisson semigroups  of the Laplacian and Hermite operator \cite{Crescimbeni, Betancor1}.
 Analogous results have  been obtained  for
  semigroups associated with Fourier-Bessel expansions \cite{Betancor2}, for spherical means or averages along curves $(t, t^a)$ in the plane \cite{Jones2} and
for differential and singular integral operators  in some weighted Lebesgue spaces \cite{Ma1, Ma2}. 
Some results are also known for the  Ornstein--Uhlenbeck semigroup; see below.

    % In this paper, we prove variational inequalities for the Ornstein--Uhlenbeck semigroup in dimension $1$.
To define the variation seminorm $v(\rho)$, let  $\phi$ be a real- or complex-valued function defined in an interval $I$. Then for  $1 \le \rho < \infty$
           %To be more precise, if  $1 \le \rho < \infty$ and the real-valued function $\phi(t)$ is defined in an interval $I$, its variation norm is defined as
\begin{equation*}
  \|\phi\|_{v(\rho), I} := \sup \left( \sum_{i=1}^{n} |\phi(t_i) - \phi(t_{i-1})|^\rho \right)^{1/\rho},
\end{equation*}
where the supremum is taken over all finite, increasing sequences $\left(t_i \right)_0^n$ of points in $I$.
         This is a seminorm which vanishes only for constant functions. We will often omit indicating the interval $I$.
           The space $V(\rho,I)$ consists of those functions $\phi$ in $I$ for which $ \|\phi\|_{v(\rho),I} < \infty$. In this paper we will only consider the variation of continuous functions  $\phi(t)$.

We next introduce the one-dimensional Ornstein-Uhlenbeck semigroup.
Let $R(x) = x^2/2$ for $x \in \mathbb R$, and define the measure  $d\gamma_\infty(u) = (2\pi)^{-1/2} \, \exp(-R(u))\,du$ in $\mathbb R$.
The  semigroup is then given by
\begin{equation*}
  H_t f(x) =  \int f(u)\,K_t(x,u)\,d\gamma_\infty(u), \qquad t>0,
\end{equation*}
where $f \in L^1(\gamma_\infty)$ and the kernel $K_t$ is
\begin{equation}\label{def:Mehler}
K_t(x,u) = \frac {e^{R(x)}}{\sqrt{1-e^{-2t}}} \, \exp\left(-\frac12\,\frac{(e^{-t}u - x)^2}{1-e^{-2t}}\right), \qquad t>0, \quad (x,u)\in \mathbb R\times\mathbb R.
\end{equation}
The measure $\gamma_\infty$ is the unique probability measure which is invariant under the semigroup.

We will consider the variation of the semigroup, i.e., the seminorm  $\| H_t f(x)\|_{v(\rho),\Bbb R_+}$ taken with respect to $t$ and considered as an operator defined  for $f \in L^1(\gamma_\infty)$.

When  $\rho > 2$
it is known that this operator is bounded from $L^p(\gamma_\infty)$ to $L^p(\gamma_\infty)$ for
                 %The boundedness on $L^p(\gamma_\infty)$ of the variation norm of this operator for
                 $1<p<\infty$, even for the Ornstein--Uhlenbeck semigroup in any finite dimension.  This follows from \cite{Jones1}, where a general symmetric diffusion semigroup  is considered.
              Another proof
 can be found in
  \cite[Corollary 4.5]{Le Merdy}; it is verified  in
 \cite[page 31]{Almeida}  that this corollary can be applied in our setting.

 The inspiration for the present work came from a comment in   \cite[page 31]{Almeida} saying that no variational weak type $(1,1)$ bound is known for the Ornstein--Uhlenbeck semigroup.         %, even in dimension one.
We will prove the following one-dimensional result.

\begin{theorem}\label{thm}
 For each  $\rho > 2$ the operator that maps   $f \in L^1(\gamma_\infty)$ to the function
  \begin{equation}\label{norm-op}
  \| H_t f(x)\|_{v(\rho),\Bbb R_+}, \quad x \in \mathbb R,
\end{equation}
where the $v(\rho)$ seminorm is taken in the variable $t$, is of weak type $(1,1)$ with respect to the measure $\gamma_\infty$.
\end{theorem}
In other words, the inequality
\begin{equation} \label{main-ineq}
\gamma_\infty
\{x\in\R :   \| H_t f(x)\|_{v(\rho),\Bbb R_+}
             %\left( \sum_{i=1}^{n} H_{t_i}f(x) - H_{t_{i-1}}f(x)|^\rho \right)^{1/\rho}
     > \alpha\} \le \frac{C}\alpha\,\|f\|_{L^1( \gamma_\infty)}, \qquad \alpha>0,
\end{equation}
 holds for some $C > 0$ and all functions $f\in L^1 (\gamma_\infty)$.
                                    %with $C=C(n,Q,B)$.

\bigskip

The structure of this paper is as follows. Section \ref{prelim} contains some preliminaries, mainly concerning the variation seminorm and the $t$ derivative of $K_t$. In the following sections, Theorem \ref{thm}  will be obtained as a direct consequence of Propositions \ref{t>1}, \ref{t<1,global} and \ref{locsmallt}. Of these, Proposition \ref{t>1}
                   % is stated and proved in Section \ref{t large} and
 deals with the variation only for $t \ge 1$. For these values of $t$, the estimate \eqref{main-ineq} is slightly strengthened. In  Section~\ref{global}, we split the operator given by the variation for $0 < t \le 1$ into a local and a global part. This is done by means of a partition of the line into intervals where  the density of $\gamma_\infty$ is essentially constant. Then Proposition~\ref{t<1,global}, dealing with the global part, is  proved.

Proposition  \ref{locsmallt} is an estimate of the local part of the variation in  $0 < t \le 1$, and is stated and proved in the long
Section \ref{local1}. The proof goes via
 Proposition \ref{prp}, which  deals with one of the intervals of the partition at a time, and where $\gamma_\infty$ is replaced by Lebesgue measure.
 Finally,  Proposition \ref{prp} is seen to follow from estimates for the variation of integrals of an  $L^1$ function over certain intervals, obtained as a consequence of a known theorem about the variation of mean values.

 Theorem \ref{thm} is proved in the one-dimensional case, and
 the handling of integrals over intervals in  Section \ref{local1} just mentioned seems hard to extend to higher dimensions, because of geometrical obstructions.  Only the results of Sections \ref{t large} and  \ref{global} extend easily.

 We point out that in Sections \ref{t large} and  \ref{global}, we use arguments similar to some from the authors' papers \cite{CCS3} and \cite{CCS5}. Rather than invoking the results from there, we prefer to give the proofs explicitly.

\vskip35pt

\section{Preliminaries}\label{prelim}

By  $C < \infty$ and $c > 0$  we denote many different  absolute constants,   and $X \lesssim Y$, or equivalently  $Y \gtrsim X$, means $X \le C Y$.
We write  $X \simeq Y$  if both  $X \lesssim Y$  and $Y \lesssim X$.

Seminorms of type $\| . \|_{v(\rho),I}$ will always be taken in one of the variables $t$ or  $\tau$.

We will let  $\dot K_t(x,u) = \partial K_t(x,u)/\partial t$.

It is not immediately obvious that the function  $x \mapsto \| H_t f(x)\|_{v(\rho),\Bbb R_+}$     is measurable. But  $H_t f(x)$ is continuous in $t$ for each $x$, as seen by dominated convergence. In the definition of the ${v(\rho)}$ seminorm, it is therefore enough to consider  sequences of points $t_i \in \mathbb{Q}$, thus only a countable family of sequences. The measurability follows.

\vskip5pt

We   give some simple properties of the variation, and first observe that the seminorm
$\|.\|_{v(\rho)}$ is  decreasing in  $\rho$ for $1 \le \rho < \infty$.
This seminorm is  also subadditive in $I$, in the following sense. Take an inner point $\tau$ of $I$ and set $I_+  = I \cap [\tau, +\infty)$ and $I_-  = I \cap (-\infty, \tau]$. Then for  $1 \le \rho < \infty$ and any $\phi$
\begin{equation*}
   \|\phi \|_{v(\rho), I} \le \|\phi \|_{v(\rho), I_+} + \|\phi  \|_{v(\rho), I_-}.
 \end{equation*}

\begin{lemma} \label{var}
Let $1 \le \rho < \infty$.

 (a) If  $\phi \in C^1(I)$ and $\phi' \in L^1(I)$, then  $\phi \in V(\rho,I)$ and
 \begin{equation*}
   \|\phi  \|_{v(\rho), I} \le \int_{I} |\phi'(t)|\,dt.
 \end{equation*}

 (b)   If $\phi$ is monotone and bounded in $I$, then  $\phi \in V(\rho,I)$ and
\begin{equation*}
  \|\phi\|_{v(\rho), I} \le 2 \sup_I |\phi|.
\end{equation*}
\end{lemma}

Both parts here are easy for $\rho = 1$ and then follow for all $\rho$.

The variation of products can be estimated as follows.

\begin{lemma} \label{var_prod}
  Let $\phi$ and $\psi$ be bounded functions defined in the interval $I$. Then
  for any $1 \le \rho < \infty$
  \begin{equation*}
 \|\phi\psi  \|_{v(\rho)} \le   \|\phi  \|_\infty   \|\psi  \|_{v(\rho)}   +   \|\phi  \|_{v(\rho)}  \|\psi\|_\infty.
\end{equation*}
\end{lemma}
To prove this, it is enough to write for an increasing sequence $(t_i)$  in $I$
  \begin{equation*}
  \phi(t_i)\psi(t_i) - \phi(t_{i-1})\psi(t_{i-1}) = \phi(t_i)(\psi(t_i) -\psi(t_{i-1})) + (\phi(t_i)- \phi(t_{i-1}))\psi(t_{i-1}),
\end{equation*}
and then take the $\ell^\rho$ norm.

\vskip4pt

We next make some preparations for the proof of Theorem \ref{thm}.

The following estimate of the variation seminorm of $H_t f(x)$  will be useful. Let the interval $I$ be either
$(0,1]$ or $[1,\infty)$. 
From Lemma \ref{var}(a), we conclude that
\begin{align}  \label{important}
 \|H_tf(x)\|_{v(\rho), I}&\le       %\int_I \left| \frac{\partial}{\partial t} H_tf(x)\right| dt \notag
    \int_I \left|\frac{\partial}{\partial t} \int K_t(x,u) f(u)\,d\gamma_\infty (u)  \right|\,dt\notag \\
 &=\int_I \left| \int \dot K_t(x,u) f(u) \,d\gamma_\infty (u) \right|\,dt  \notag
\\     &\le \int         \int_I \big|  \dot K_t(x,u)\big| \,dt \, |f(u)|\,  d\gamma_\infty (u).
  \end{align}
To justify moving the differentiation inside the integral in the second step here, we refer to
\cite[Lemma~5.3]{CCS5}.

\vskip4pt

We compute and estimate $\dot K_t(x,u)$.

\begin{lemma}\label{derivate-nucleo}
For all $(x,u) \in \mathbb R^n\times \mathbb R^n$ and
$t>0$, we have
\begin{align}\label{P}
\dot K_t(x,u) =&
\, K_t(x,u)\, \left(
 -\frac {e^{-2t}}{1-e^{-2t}}+\frac{e^{-2t}(e^{-t}u - x)^2}{(1-e^{-2t})^2}+ \frac{e^{-t} u (e^{-t}u - x)     }{1-e^{-2t}}\right).
 \end{align}
Moreover, for $t\ge 1$ one has
 \begin{align}\label{P1}
\big| \dot K_t(x,u) \big|\lesssim e^{R(x)}
\exp
\big(
-c
\left(
e^{-t}\,u- x\right)^2
\big)\big(
e^{-t}\,|u|
 +e^{-2t}\big).         % \quad  \text{ for \hskip6pt $t\geq1$.}
\end{align}
\end{lemma}
\begin{proof}
We omit the proof of \eqref{P}. When $t\ge 1$, \eqref{P}  implies that
 \begin{align}\label{P2}
\big| \dot K_t(x,u) \big|\lesssim K_t(x,u) \big(
|e^{-t}\,u- x|\,e^{-t}\,|u|+
e^{-2t}\,
(e^{-t}\,u- x)^2 +e^{-2t}\big).         % \quad  \text{ for \hskip6pt $t\geq1$.}
\end{align}
For $t \ge 1$ \eqref{def:Mehler} shows that  $K_t \lesssim e^{R(x)}\,\exp\big(-c\left(e^{-t}\,u- x\right)^2\big) $. Changing the value of $c$ here,
 we may neglect the factors $e^{-t}\,u- x$ in \eqref{P2} and obtain
\eqref{P1}.
\end{proof}

\vskip35pt

\section{The case of large $t$}\label{t large}

In this section we consider the variation of  $H_tf(x)$ only for $1 \le t < \infty$.

\begin{proposition} \label{t>1}
 For each  $\rho > 2$ the operator that maps   $f \in L^1(\gamma_\infty)$ to the function
  \begin{equation*}
  \| H_t f(x)\|_{v(\rho), [1,+\infty)}, \quad x \in \mathbb R,
\end{equation*}
is of weak type $(1,1)$ with respect to the measure $\gamma_\infty$.
In fact, one has the following stronger result: If  $\rho > 2$ and $\|f\|_{L^1( \gamma_\infty)} = 1$, then
\begin{equation} \label{stronger}
\gamma_\infty
\left\{x\in\R:   \| H_t f(x)\|_{v(\rho),[1,\infty)}
        > \alpha \right\} \lesssim \frac{1}{\alpha \sqrt{\log \alpha}}, \qquad \alpha > 2.
\end{equation}
\end{proposition}

Notice
that, when $t$ is large, the estimate \eqref{main-ineq}  is enhanced by a logarithmic factor. In \cite{CCS2} and \cite{CCS3},  an analogous  phenomenon was
already  observed  both for the Ornstein--Uhlenbeck maximal operator and for the Gaussian Riesz transform.

\begin{proof}
  %[of Proposition \ref{t>1}]{proof}
 Let
 $f$  be normalized in $L^1( \gamma_\infty)$.
 We integrate  \eqref{P1}, getting
 \begin{align*}
 \int_{1}^{\infty}  \big| \dot K_t(x,u) \big|   \, dt    \lesssim e^{R(x)}
 \int_{1}^{\infty} \exp \big( -c \left(e^{-t}\,u- x\right)^2\big)\big(e^{-t}\,|u| +e^{-2t}\big) \,dt.
\end{align*}
In the last parenthesis in the second integrand here, we consider first only the term $e^{-t}\,|u|$ and make the
 change of variable $e^{-t}\,u- x=y$, separately for $u>0$ and $u<0$. As a result,
\begin{align*}
\int_1^\infty
\exp\big({
-c
(
e^{-t}\,u- x)^2
}\big)\,
e^{-t}\,|u| \,dt
\le\int_{\R}\exp(-cy^2)\, dy \simeq 1.
\end{align*}
Taking also the term $e^{-2t}$ in the integral above into account, we conclude that
\begin{align*}
 \int_{1}^{\infty}  \big| \dot K_t(x,u) \big|   \, dt    \lesssim e^{R(x)}.
 \end{align*}

 Now \eqref{important} leads to
 \begin{align*}
  \|H_tf(x)\|_{v(\rho), [1,\infty)}   \lesssim e^{R(x)}.
 \end{align*}
 It is easily seen that
 \begin{equation*}
   \gamma_\infty\left\{x : e^{R(x)} > \beta\right\}    \lesssim \frac{1}{\beta  \sqrt {\log \beta}}, \qquad \beta > 2.
 \end{equation*}
 From this    \eqref{stronger} follows,  and  since  \eqref{main-ineq} is trivial for  $\alpha \le 2$,  Proposition~\ref{t>1} is proved.
 \end{proof}

\vskip35pt

   \section{The global case with small $t$}\label{global}

  We first split the operator $H_t$ in a local and a global part, in a way adapted to   $\gamma_\infty$.
  Let  $\eta \ge 0$ be  a smooth function in $\Bbb R_+$ which is 1 in $(0,1/2]$ and 0 in $[1,\infty)$.
The local part of the semigroup is defined by
\begin{equation*}
  H_t^{\mathrm{loc}}  f(x) =  \int f(u)\, K_t(x,u)\,\eta((1+|x|)|x-u|)\,d\gamma_\infty(u).
\end{equation*}
The global part $H_t^{\mathrm{glob}} = H_t -  H_t^{\mathrm{loc}}$ is given by a similar expression, with
$\eta(.)$ replaced by $1 - \eta(.)$.

 \begin{proposition} \label{t<1,global}
 For each  $\rho > 2$ the operator that maps   $f \in L^1(\gamma_\infty)$ to the function
  \begin{equation*}
  \| H_t^{\mathrm{glob}} f(x)\|_{v(\rho), (0,1]}, \quad x \in \mathbb R,
\end{equation*}
is of weak type $(1,1)$ with respect to the measure $\gamma_\infty$.
\end{proposition}

  \begin{proof}
    We first give an estimate of the number of zeros of the function $t \mapsto \dot K_t(x,u)$ for $0 < t < 1$.
    From  \eqref{P} we see that we can write
     \begin{equation*}
   \dot K_t(x,u) = K_t(x,u)\, \frac{P_{x,u}(e^{-t})}{(1-e^{-2t})^2},
  \end{equation*}
  where $P_{x,u}$ is a polynomial of degree at most 4, with coefficients depending on  $x$ and $u$. Thus  $\dot K_t(x,u)$ can have at most four zeros in $(0,1)$.
     Denote these zeros by $t_1,\dots, t_{N-1}$; the $t_i$ and also $N$ will depend on $(x,u)$, and $N \le 5$. Set also
$t_0 = 0$ and $t_N = 1$.
Then
  \begin{equation*}
\int_0^1  |\dot K_t(x,u)|\,dt = \sum_1^N \left| \int_{t_{i-1}}^{t_i} \dot K_t(x,u)\,dt \right| \le
10\, \sup_{(0,1]}K_t(x,u).
  \end{equation*}

Since the computation \eqref{important} remains valid with an extra factor $1-\eta((1+|x|)|x-u|)$, we conclude
 \begin{equation} \label{33}
\| H_t^{\mathrm{glob}} f(x)\|_{v(\rho), (0,1]} \lesssim
\int     |f(u)| \sup_{(0,1]}K_t(x,u)\,(1-\eta((1+|x|)|x-u|))\,d\gamma_\infty(u)
 \end{equation}

We claim that for $0<t \le 1$ and all $(x,u)$
\begin{equation}  \label{supKt}
 \sup_{(0,1]}K_t(x,u)\,(1-\eta((1+|x|)|x-u|)) \lesssim e^{R(x)}\,(1+|x|).
  \end{equation}

  If $1-\eta((1+|x|)|x-u|) \ne 0$, we have  $|x-u| >  1/(2(1+|x|))$ and thus for $0<t\le 1$ also
    \begin{align*}
 (1+|x|)^{-1} & < 2 |x-u| \le 2|x-e^{t}\,x| + 2|e^{t}\,x - u |
  = 2 (e^{t}-1)|x| + 2e^{t} |x-e^{-t}u|   \\
 & \le 2et (1+|x|) + 2e|x-e^{-t}u|.
  \end{align*}
  Now, if $t (1+|x|)^2 < 1/(4e)$ we get a bootstrap implying
  \begin{equation*}
 (1+|x|)^{-1}  <  4e|x-e^{-t}u|.
  \end{equation*}
  Then we see from \eqref{def:Mehler} that
     \begin{align*}
 e^{-R(x)}\,K_t(x,u) & \simeq t^{-1/2}\, \exp \left( -\frac12 \frac{(e^{-t}u-x)^2}{t}\right)  \\ &\le
       t^{-1/2}\, \exp \left(-\frac12\, \frac1{16e^2t(1+|x|)^2} \right)  \lesssim 1+|x|,
  \end{align*}
  and \eqref{supKt} follows. On the other hand, if $t (1+|x|)^2 \ge 1/(4e)$,   \eqref{supKt} also follows, since
  then $t^{-1/2} \lesssim 1+|x|$.

  Combining now \eqref{33} and \eqref{supKt}, we get
  \begin{equation*}
\| H_t^{\mathrm{glob}} f(x)\|_{v(\rho), (0,1]} \lesssim  e^{R(x)}\,(1+|x|) \, \| f \|_{L^1(\gamma_\infty)}.
 \end{equation*}
This ends the proof of Proposition \ref{t<1,global}, because
\begin{equation*}
   \gamma_\infty\left\{x: \,   e^{R(x)}\,(1+|x|) > \beta \right\} \lesssim  \frac1\beta,  \qquad \beta > 0.
\end{equation*}

  \end{proof}

\vskip4pt

\vskip35pt

   \section{The local case with small $t$}\label{local1}

This section consists of the proof of the following result.

   \begin{proposition} \label{locsmallt}
     For each  $\rho > 2$ the operator that maps   $f \in L^1(\gamma_\infty)$ to the function
  \begin{equation*}
  \| H_t^{\mathrm{loc}} f(x)\|_{v(\rho),(0,1]}, \quad x \in \mathbb R^n,
\end{equation*}
is of weak type $(1,1)$ with respect to the measure $\gamma_\infty$.
   \end{proposition}

\vskip25pt

\subsection*{Splitting of the line into local intervals}\label{subs1}

    ~

     \vskip4pt

            %The proof of this proposition will occupy this section and the next.

 The localization means that the value $H_t^{\mathrm{loc}}  f(x)$ depends only on the restriction of $f$ to the interval $\{u:\:|u-x| \le 1/(1+|x|)\}$, and we will split the line into intervals of similar type.
 Choose an increasing sequence $(x_j)_0^\infty$ with $x_0 =0$ such that for $j = 0,1,\dots$
 \begin{equation*}
x_{j+1} - \frac{1}{1+x_{j+1}} = x_{j} + \frac{1}{1+x_{j}}.
 \end{equation*}
This recursion formula determines the sequence uniquely. We have $x_{j+1} -x_{j} < 2$ for all $j \ge 0$, so that $x_j \le 2j$. Thus $x_{j+1} -x_{j} \gtrsim 1/j $ and $x_j \to +\infty$ as $j \to \infty$. (In fact, $x_{j}$ is close to $2\,\sqrt {j}-1$,  as shown in the Appendix.) For $j<0$ we let
$x_j = - x_{|j|}$.

 The intervals
 \begin{equation*}
I_j = \left[ x_{j} - \frac{1}{1+|x_{j}|}, \: x_{j} + \frac{1}{1+|x_{j}|}\right], \qquad j\in \Bbb Z,
 \end{equation*}
 are pairwise disjoint except for endpoints, and they cover $\Bbb R$.
 If $\mathrm{supp}\, f \subset I_j$, we claim that the support of $H_t^{\mathrm{loc}} f $ is contained in the interval
 \begin{equation*}
\widetilde I_j = \left[ x_{j} - \frac{4}{1+|x_{j}|}, \: x_{j} + \frac{4}{1+|x_{j}|}\right].
 \end{equation*}
 To verify this, let $x \in \mathrm{supp}\,H_t^{\mathrm{loc}} f$. Then $x$ has distance at most $1/(1+|x|)$ from some point in $I_j$,  so that
 \begin{equation}\label{x-xj}
   |x-x_j| \le\frac 1{1+|x|} + \frac 1{1+|x_j|},
 \end{equation}
  which implies
 \begin{equation*}             %\label{comp}
   1+|x| \ge 1+|x_j| - \frac 1{1+|x|} - \frac 1{1+|x_j|} \ge |x_j| - 1 \ge \frac {1+|x_j|}3,
 \end{equation*}
 the last inequality holding only  if $|x_j| \ge 2$. But if  $|x_j| < 2$, one has the same estimate, since then
 $1+|x| \ge 1 \ge ( {1+|x_j|})/3$. The claim now follows from \eqref{x-xj}.

We also observe that if  $\mathrm{supp} \,f \subset I_j$ and $j>0$, then
 $\mathrm{supp} \,H_t^{\mathrm{loc}}f \subset \{x \ge 0\} $. Indeed, $I_j \subset [1,\infty)$ when
  $j>0$, so if $x \in \mathrm{supp} \,H_t^{\mathrm{loc}}f$ we must have $x \ge 1 - 1/(1+|x|) \ge 0$.

   % if $u \in I_j$ and $x<0$, then \begin{multline*}    |u-x| = |x| +u \ge |x|+ x_1 - 1/(1+x_1) =|x|+ x_0 + 1/(1+x_0) \\  = |x|+  1 > 1/(1+|x|),  \end{multline*} and so $x \notin \mathrm{supp} \,H_t^{\mathrm{loc}}f$.

   The intervals $\widetilde I_j$ have bounded overlap. Therefore, it is enough to prove Theorem~\ref{thm} for functions $f$ supported in $I_j$, with a bound that is uniform in $j \in \Bbb Z$.

In each $\widetilde I_j$, the density          %$(2\pi)^{-1/2}\,e^{-R(x)}$
of $\gamma_\infty$ is essentially constant,
since $e^{-R(x)} \simeq e^{-R(x_j)}$ for $x \in \widetilde I_j$,
 and this is uniform in $j$. Therefore, we can pass to Lebesgue measure in $u$ and in $x$.  We replace
$f \in L^1(\gamma_\infty)$, supported in $I_j$, by $g(u) = f(u)\,e^{-R(u)} \in L^1(du)$, with the same support. Instead of
$H_t^{\mathrm{loc}}$, we can then consider the   operator
\begin{equation*}
\mathcal  H_t^{\mathrm{loc}}  g(x) = \frac {1}{\sqrt{1-e^{-2t}}} \, \int g(u)\, \exp\left(-\frac12 \, \frac{(e^{-t}u-x)^2}{1-e^{-2t}}\right)\,\eta((1+|x|)|x-u|)\,du,
\end{equation*}
where we deleted the essentially constant factor  $e^{R(x)}$.

We conclude from the above that the following proposition implies Proposition~\ref{locsmallt}.                   %%%Theorem \ref{thm}.

\begin{proposition}\label{prp}
  For $\rho >2$ and each $j\in \Bbb Z$, the operator that maps  $g \in L^1(I_j)$ to
 \begin{equation*}
  \| \mathcal H_t^{\mathrm{loc}}  g(x)\|_{v(\rho), (0,1]}
  \end{equation*}
 is bounded from $L^1(I_j)$ to $L^{1,\infty}(\widetilde I_j)$,
 where the intervals are endowed with the  Lebesgue measure. This is  uniform in $j$.
\end{proposition}

\vskip25pt

\subsection*{Proof of Proposition \ref{prp}}\label{subs2}

~

 \vskip4pt
                      %\section{Proof of Proposition \ref{prp},  Part I}\label{local2}

              %\begin{proof}                            %%   [of Proposition \ref{prp}]{proof}
  For  symmetry reasons, it is enough to consider only $j \ge 0$ and only points $x \in \Bbb R_+$.
 With  $j \ge 0$ fixed, we let $g \in L^1(I_j)$. The expression for  $\mathcal  H_t^{\mathrm{loc}}  g(x)$ will be rewritten in terms of integrals of only $f(u)$ over many intervals which depend on $t$. Here we follow
  \cite[proof of Lemma 2.4]{Campbell}, writing
  \begin{equation*}
\exp (-y^2/2) = - \int_{y}^{\infty} \frac{d e^{-s^2/2}}{ds}\,ds = - \int_{0}^{\infty} \chi_{y<s}\:\frac{d e^{-s^2/2}}{ds}\:ds
 \end{equation*}
and
\begin{equation*}
\eta (y) = - \int_{y}^{\infty}\frac{d \eta(\sigma)}{d\sigma}\,d\sigma =
 - \int_{1/2}^{1}  \chi_{y<\sigma}\: \frac{d \eta(\sigma)}{d\sigma}\:d\sigma.
 \end{equation*}

 As a result,
 \begin{equation} \label{Ht}
  \mathcal  H_t^{\mathrm{loc}} f(x) =
    \int_{0}^{\infty}\frac{d e^{-s^2/2}}{ds}\,    \int_{1/2}^{1}\frac{d \eta(\sigma)}{d\sigma}\,
 R_t^{s,\sigma}g(x) \;d\sigma\,ds,
 \end{equation}
  where
  \begin{equation}\label{Rt}
    R_t^{s,\sigma}g(x) = \frac 1{\sqrt{1-e^{-2t}}} \,   \int g(u)\,  \chi_{|e^{-t}u-x|/\sqrt{1-e^{-2t}}<s}\: \chi_{(1+|x|)|x-u|<\sigma}\,\,du.
  \end{equation}

Observe that it is not enough to prove that the $v(\rho)$ seminorm of $R_t^{s,\sigma}g(x)$, taken with respect to $t$, defines an operator of weak type (1,1) for each $s$ and $\sigma$. This is because $L^{1,\infty}$ is not a normed space. Instead we will estimate the variation of $R_t^{s,\sigma}g(x)$ for all $s$ and $\sigma$ in terms of one operator of weak type (1,1) (actually a small number of such operators and actually with a factor $s+1$, which is integrable against $de^{-s^2/2}/ds$).

A few times below, we will use  the simple inequalities
\begin{equation}\label{simple}
  y \le e^y - 1 \le 4y \qquad   \mathrm{for}    \qquad   0 \le y \le 2.
\end{equation}
The second inequality holds because the function $(e^y - 1)/y$ is increasing for these $y$, as seen from the power series.

In the sequel, we fix an $x \in \widetilde{I}_j \cap \mathbb R_+$ and let $s>0$ and $1/2<\sigma < 1$, but we temporarily allow all $t>0$.
We will soon introduce many quantities which will depend on $x,\,s,\,t $  and sometimes $\sigma$; in order not to make the notation too heavy, we will systematically omit indicating the dependence on $x$.

Since the inequality  $|e^{-t}u-x|/\sqrt{1-e^{-2t}}<s$ can be rewritten as
    %\begin{equation}\label{uno}
   $ |u-e^t x|< s\,\sqrt{e^{2t}-1},$
  %\end{equation}
the integration in \eqref{Rt} is taken over the interval
\begin{equation*}
   J_t(s,\sigma) = \left\{u \in \mathbb{R}: |u-e^t x|< s\,\sqrt{e^{2t}-1}\qquad \mathrm{and} \qquad |u-x| < \frac \sigma {1+x} \right\}.
\end{equation*}

Observe that $J_t(s,\sigma)$ is nonempty precisely when
\begin{equation} \label{nonzero}
  s\,\sqrt{e^{2t}-1} + \frac\sigma {1+x}  > e^t x -x,
\end{equation}
or equivalently
$Q_{s}(t) < \sigma/(1+x)$,
   %\begin{equation*} Q_{s}(t)> \sigma/(1+x) \}  -x,  \end{equation*}
where
\begin{equation}\label{defq}
Q_{s}(t) =   x(e^t-1) - s\,\sqrt{e^{2t}-1}\,.
\end{equation}
An instance of this function is plotted in Figure 1.

\vskip1cm
%%\begin{comment}
\begin{tikzpicture}
\begin{axis}
[   axis lines = left,
    xmin = -0.3, xmax = 1.5,
    ymin = -1, ymax = 2.0,
    xtick distance = 3,
    ytick distance = 3,
 xlabel = $t$,
 ylabel = {$Q_s(t)$},
         axis x line=middle, axis y line=middle]
         \draw[-] (0.15,-0.05) -- (0.15,0.05);
         \node at (0.2, 0.2) {\color{black}{\scriptsize{$\tilde t(s)$}}};
         \draw[-] (-0.05,0.25) -- (0.05 ,0.25);
          \node at (-0.15, 0.3) {\scriptsize{$\frac{\sigma}{1+x}$}};
           \draw[-] (0.7,-0.05) -- (0.7,0.05);
        \node at (0.8, -0.2) {\color{black}{\scriptsize{$t_1(s,\sigma)$}}};
           \draw[-] (0.52,-0.05) -- (0.52,0.05);
  \node at (0.52, 0.2) {\color{black}{\scriptsize{$ t_0(s)$}}};
  \addplot[
        domain = 0:1.5,
        samples = 200,
        smooth,
        thick
    ]
      {2*( exp(\x)-1)-sqrt(exp(2*\x)-1)};
\end{axis}
  \end{tikzpicture}

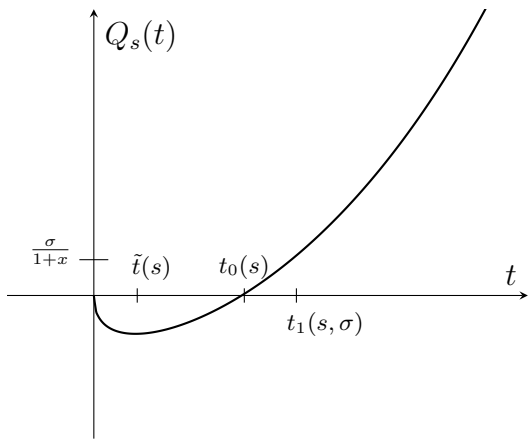
  {\bf{Figure 1.}} {Graph of $Q_s(t)$ with $x=2$, $s=1$. Here $\sigma = 3/4$.}
%%\end{comment}

%\begin{figure}[ht!]
%\centering
%\includegraphics%[width=18cm]
%{variaz13_cropped.pdf}
%\end{figure}

\bigskip
\vskip0.5cm

Since $Q'_{s}(t) = x e^t - se^{2t}/\sqrt{e^{2t}-1}  $, one finds by squaring each of these two terms that $Q'_{s}(t) > 0$ if and only if $x>s$ and $t > \widetilde t(s)$, where $\widetilde t(s) >0$ is determined by $e^{2\widetilde t(s)} = x^2/(x^2 - s^2)$. If   $x \le s$, we set   $\widetilde t(s) = +\infty$. It follows that $Q_{s}(t)$ is strictly
\begin{equation}\label{Qmon}
    \left\{ \begin{array}{ll}
       \mbox{decreasing in}\; \;&0 < t < \widetilde t(s) \\
      \mbox{increasing in}\;\; &\widetilde t(s) < t < +\infty.
    \end{array}
    \right.
\end{equation}

Further, $Q_{s}(0) = 0$, and if  $x>s$, then  $Q_{s}(t) \to +\infty$ as   $t \to +\infty$.
We conclude that there exists a   $t_1(s,\sigma) \in (0, +\infty]$ such that
\begin{equation}\label{nonempty}
J_t(s,\sigma) \ne \emptyset \quad \Leftrightarrow \quad Q_{s}(t) < \sigma/(1+x)  \quad \Leftrightarrow  \quad  0 < t < t_1(s,\sigma).
\end{equation}
Moreover,  $t_1(s, \sigma) < +\infty$ if and only if   $x > s$.

For later use, we make a similar observation regarding the inequality  $Q_{s}(t) < 0$. There exists a
$t_0(s) \in \left(\widetilde t(s), t_1(s,\sigma)\right) \cup \{+\infty\}$, finite if and only if  $x > s$,   such that
\begin{equation}\label{Qneg}
 Q_{s}(t) < 0  \quad \Leftrightarrow \quad 0 < t < t_0(s).
\end{equation}
(Actually, $t_0(s)$  is given by $e^{t_0(s)} =  \frac{x^2+s^2}{x^2-s^2}$
if $x>s$.)
                         %For $x>s$ we notice that $\widetilde t(s) < t_0(s) < t_1(s,\sigma)$.

 We let
\begin{equation}\label{defT}
T =   T(s,\sigma) := 1\wedge t_1(s,\sigma) = \sup \{t \in (0, 1]:\: Q_{s}(t) < \sigma/(1+x) \}.
\end{equation}
Then $ T(s,\sigma)  \in (0, 1]$, and
from now on we consider only $0 < t \le 1$.

  Notice that $ T(s,\sigma) < 1$ if and only if  $Q_{s}(1) > \sigma/(1+x)$.
              %The set for which   $Q_{s}(t) < \sigma/(1+x)$, i.e.,$J_t(s,\sigma) \ne \emptyset$, is given by
 Further,
  \begin{equation} \label{TT}
      \{t \in (0, 1]:\: J_t(s,\sigma) \ne \emptyset \}
      =
    \left\{ \begin{array}{ll}
      (0,T(s,\sigma)) & \mbox{if $Q_{s}(1) \ge \sigma/(1+x)$}\\
      (0, T(s,\sigma)] = (0,1] & \mbox{otherwise.}
    \end{array}
    \right.
\end{equation}

  In the first case here, $J_{T(s,\sigma)}(s,\sigma) = \emptyset$ and
  $R_{T(s,\sigma)}^{s,\sigma}g(x) = 0$.  We observe that  $R_t^{s,\sigma}g(x)$ is in all cases defined and continuous as a function of $t$ for $0<t \le T(s,\sigma)$.

Next, we deduce a bound for $T(s,\sigma)$. Since  $T(s,\sigma) \le t_1(s,\sigma)$, any $t < T(s,\sigma)$ satisfies \eqref{nonzero}. Using first  \eqref{simple} and  then   \eqref{nonzero}
             multiplied by $x$ together with \eqref{simple}, and finally
        the inequality between the geometric and arithmetic means, we get
 \begin{equation*}
    x^2\,t \le x^2\,(e^{t}-1) < sx\,\sqrt{8t} + \sigma\,\frac {x} {1+x} \le \frac{x^2t}2 + 4s^2 +1.
   \end{equation*}
    Hence,
  \begin{equation} \label{T}
x^2\, T(s,\sigma) \le 8(s^2 +1).
\end{equation}

When  $J_t(s,\sigma)$ is nonempty, we write its endpoints as
\begin{equation*}
J_t(s,\sigma) = (k_t^-(s,\sigma), k_t^+(s,\sigma)),
\end{equation*}
and they are
\begin{equation}\label{k+}
  k_t^+(s,\sigma) =  \left( xe^t + s\,\sqrt{e^{2t}-1} \right)
     \wedge \left( x + \frac{\sigma}{1+x}\right)
\end{equation}
and
\begin{equation}\label{k-}
  k_t^-(s,\sigma) =  \left( xe^t - s\,\sqrt{e^{2t}-1} \right)
     \vee \left( x - \frac{\sigma}{1+x}\right).
\end{equation}

From the last expression,  it follows  that

\begin{equation}\label{k-<x}
k_t^-(s,\sigma) < x  \quad \Leftrightarrow     \quad   Q_{s}(t) < 0    \quad
\Leftrightarrow    \quad       t < t_0(s);
   % k_t^-(s,\sigma) =  \left( xe^t - s\,\sqrt{e^{2t}-1} \right)    \vee \left( x - \frac{\sigma}{1+x}\right).
\end{equation}
see \eqref{Qneg}.

   % $k_t^-(s,\sigma) < x$  if and only if $Q_{s}(t) < 0$ or equivalently   $t < t_0(s)$; see \eqref{Qneg}.

\vskip3pt

                        %\subsection*{Proof of Proposition \ref{prp}, Part II}\label{subs3}

           %  \section{Proof of Proposition \ref{prp}, Part II}\label{local3}

      %We continue the proof of Proposition \ref{prp}, still keeping $x$ fixed.
The next step in the proof of  Proposition \ref{prp} will be to
 apply the following theorem, obtained in the discrete setting in \cite{Jones}. It can easily be transferred to the setting of $\Bbb R$, see \cite[proof of Lemma 2.1]{Campbell}.  Define the one-sided mean values of a function $\phi \in L^1(\Bbb R)$ by
 \begin{equation*}
    M_\tau^+\,\phi(x) = \frac 1 \tau \, \int_{x}^{x+\tau} \phi(u)\,du, \qquad x \in \Bbb R,\;\; \tau >0,
  \end{equation*}
  and $M_\tau^-\,\phi(x)$ similarly using the interval $(x-\tau,x)$.

\begin{theorem} \label{jones}
 (\cite[Theorem 3.6]{Jones}) For $2<\rho<\infty$, the operator that maps    $f \in L^1(\Bbb R)$ to the function
   \begin{equation*}
   \| M_\tau^+\,\phi(x) \|_{v(\rho),\Bbb R_+}, \qquad x \in \mathbb R,
   \end{equation*}
   is of weak type (1,1) with respect to Lebesgue measure in $\Bbb R$.
   Here the variation is taken in the variable $\tau$.
\end{theorem}

This clearly holds also with $M_\tau^+$ replaced by
$M_\tau^-$.

Thus we need to rewrite  $   R_t^{s,\sigma}g(x) = (1-e^{-2t})^{-1/2} \,   \int_{J_t(s,\sigma)} g(u)\,  du$ in terms of mean values of $g$ in intervals with one endpoint at $x$.

With $0<t\le T(s,\sigma)$, we define
$J_t^+(s,\sigma) = (x,k_t^+(s,\sigma))$, which is an  interval  of length
\begin{equation}\label{lengthj+}
|J_t^+(s,\sigma)| = k_t^+(s,\sigma) - x =\left( (e^t-1)x + s\,\sqrt{e^{2t}-1} \right) \wedge \frac\sigma{1+x}\,.
\end{equation}
We further define $J_t^-(s,\sigma)  =   (k_t^-(s,\sigma),x)$, considered as an oriented interval in the sense that if $x < k_t^-(s,\sigma)$,  an integral over ${J_t^-(s,\sigma)}$ is interpreted as
minus the integral over
$(x, {k_t^-(s,\sigma)})$.
Its length is
\begin{align}\label{lengthj-}
|J_t^-(s,\sigma)|
 = & \,|k_t^-(s,\sigma) - x|\\ =&\,
     \left| x(e^t-1) - s\,\sqrt{e^{2t}-1} \right|
     \wedge \frac{\sigma}{1+x} =
   |Q_{s}(t)|  \wedge   \frac{\sigma}{1+x}\,,
   \end{align}
   as follows from the expression for $k_t^-(s,\sigma)$ if one separates the cases when the equivalent statements of
    \eqref{k-<x} are true or false.
    %Indeed, for $k_t^-(s,\sigma) \ge x$ \eqref{k-<x} says that $Q_{s}(t) \ge 0$, and \eqref{lengthj-}  follows.
     %If instead $k_t^-(s,\sigma) < x$, then $Q_{s}(t) < 0$ and we get  $|J_t^-(s,\sigma)| = (-Q_{s}(t))\wedge \sigma/(1+x)$, which is  \eqref{lengthj-}.

We now have for any $0<t\le T(s,\sigma)$
\begin{align}\label{means}
   R_t^{s,\sigma}g(x) = &\, \frac 1{\sqrt{1-e^{-2t}}} \, \int_{J_t^+(s,\sigma)}g(u)\,  du +
   \frac 1{\sqrt{1-e^{-2t}}} \, \int_{J_t^-(s,\sigma)} g(u)\,  du   \notag\\
 = & \, \frac{|J_t^+(s,\sigma)|}{\sqrt{1-e^{-2t}}}\, M_{|J_t^+(s,\sigma)|}^+ \, g(x) \pm
  \frac{|J_t^-(s,\sigma)|}{\sqrt{1-e^{-2t}}}\, M_{|J_t^-(s,\sigma)|}^\mp \,g(x),
\end{align}
where one should read the upper signs in $\pm$ and $\mp$ if $ k_t^-(s,\sigma) < x$ and otherwise the lower signs.
Notice that the two terms here cancel for $t = T(s,\sigma)$ if $J_{T(s,\sigma)}(s,\sigma) = \emptyset$,
since then $Q_s(T(s,\sigma)) = \sigma/(1+x)$ and so
 $k_t^+(s,\sigma) = k_t^-(s,\sigma) = x + \sigma/(1+x)$.

We will next
consider the variation in  $0<t\le T(s,\sigma)$ of the two mean values in \eqref{means}, and start with $M_{|J_t^+(s,\sigma)|}^+ \,g(x)$.  Since $|J_t^+(s,\sigma)|$ is a nondecreasing, continuous function of $t$ in this interval, we can reparametrize
 $M_{|J_t^+(s,\sigma)|}^+ \,g(x),\;0<t\le T(s,\sigma) $,  as $M_\tau^+ \,g(x)$
    % \begin{equation}\label{mean}  M_\tau^+\varphi(x)  = \tau^{-1}  \int_{x<u <x+\tau\varphi(u)\,du,  \end{equation}
  with $0 < \tau \le \tau_0$ for some $\tau_0 = \tau_0(s,\sigma)$.
  This reparametrization does not change the variation, so that
   \begin{equation*}
 \|M_{|J_t^+(s,\sigma)|}^+ \,g(x) \|_{v(\rho), (0,T]} =
   \|M_\tau^+\, g(x)\|_{v(\rho), (0,\tau_0]}
    \end{equation*}
 for each $s$ and $\sigma$, with the variations taken with respect to $t$ and $\tau$, respectively.
Extending the range of $\tau$ here, we conclude that
 \begin{equation} \label{weak+}
 \|M_{|J_t^+(s,\sigma)|}^+ \,g(x)\|_{v(\rho), (0, T]} \le
   \|M_\tau^+\, g(x)\|_{v(\rho), \Bbb R_+}.
    \end{equation}
                   % (Recall that $g \in L^1(\Bbb R)$ with compact support.)
    The right-hand side here is independent of  $s$ and $\sigma$, and   Theorem \ref{jones} applies to it.
           %because of the $\Bbb R$ version of  Theorem 3.6 in Jones et al.

To deal with  $M_{|J_t^-(s,\sigma)|}^\mp \,g(x)$,
we first consider the case when $t_0(s) < T(s,\sigma)$.
 At the point $t = t_0(s)$, the difference
 $k_t^-(s,\sigma) - x$  changes sign, and $|J_{t_0(s)}^-(s,\sigma)| = 0$.
Observe that if $x$ is a Lebesgue point for $g$, then  $M_{|J_t^-(s,\sigma)|}^+ \,g(x)$
 tends to $g(x)$ as  $t \downarrow t_0(s)$ and similarly for
 $M_{|J_t^-(s,\sigma)|}^- \,g(x)$ as $t \uparrow t_0(s)$. Then the second factor in the last term of \eqref{means} will be continuous in $t$ also at $t = t_0(s)$, if interpreted as  $g(x)$  at this point. This last term, with the $\pm$ sign, is also continuous, because the first factor is continuous and vanishes at  $t = t_0(s)$. We will consider the variation of    $M_{|J_t^-(s,\sigma)|}^\mp \,g(x)$
  separately in the subintervals $0 < t \le  t_0(s)$ and  $t_0(s) \le t \le T(s,\sigma)$.

 To obtain subintervals where the length $|J_t^-(s,\sigma)|$ is monotone, we invoke \eqref{Qmon} and split $(0,  t_0(s)]$ further into
 $\left(0, \widetilde t(s)\right]$  and  $\left(\widetilde t(s),  t_0(s)\right]$.
We can now  reparametrize as before in each of the three subintervals of $(0, T(s,\sigma)]$ obtained. The only little difference is that $\tau$ may now run in an interval that stays away from 0, but we can still extend its range to $\Bbb R_+$. We conclude that for every Lebesgue point $x \in \widetilde{I}_j \cap \mathbb R_+$, thus for a.a.~$x \in \widetilde{I}_j \cap \mathbb R_+$,
\begin{equation} \label{weak-}
 \|M_{|J_t^-(s,\sigma)|}^\mp \,g(x)\|_{v(\rho), (0, T]} \le
  2\, \|M_\tau^-\, g(x)\|_{v(\rho), \Bbb R_+} + \,\|M_\tau^+\, g(x)\|_{v(\rho), \Bbb R_+};
    \end{equation}
    here and below  $T = T(s,\sigma)$.
  This ends the case  $t_0(s) < T(s,\sigma)$.

  The remaining case  $t_0(s) \ge T(s,\sigma)$ is slightly easier. Then $T(s,\sigma) = 1$, and $k_t^-(s,\sigma) < x$ for $t < 1$.
  If  $\widetilde t(s) < 1$, we split  $(0,  1]$ into
    $\left(0, \widetilde t(s)\right]$  and
      $\left[\widetilde t(s), 1\right]$; otherwise no splitting is necessary.
       When  $t_0(s) = 1$, we again need to assume that  $x$ is a  Lebesgue point.
       It follows that \eqref{weak-} holds also in this case.

Since we are going to apply Lemma \ref{var_prod} to the products in \eqref{means}, we  observe that the $L^\infty$ norms          %, with respect to the variable $t$,
 of the means in \eqref{means} are controlled by standard maximal operators of $g$.
 More precisely,
\begin{equation}\label{maxfcn+}
  \|M_{|J_t^+(s,\sigma)|}^+ \,g(x)\|_{L^\infty} \le \mathcal M^+ g(x),
\end{equation}
\begin{equation}\label{maxfcn-}
  \|M_{|J_t^-(s,\sigma)|}^+ \,g(x)\|_{L^\infty}  \le \mathcal M^+ g(x),
\end{equation}
and
\begin{equation}\label{maxfcn-}
  \|M_{|J_t^-(s,\sigma)|}^- \,g(x)\|_{L^\infty}  \le \mathcal M^- g(x),
\end{equation}
where the $L^\infty$ norms are taken with respect to $t$, and
\begin{equation*}
   \mathcal M^\pm g(x) = \sup_{\tau>0} M_\tau^\pm |g|(x).
\end{equation*}

It remains to deal with the two factors in front of  the mean values in  \eqref{means}. They are
\begin{align}
F_\pm :=  \frac{|J_t^\pm(s,\sigma)|}{\sqrt{1-e^{-2t}}} = &\label{factors1}
\,\frac{\left|x(e^t - 1)\pm s\sqrt{e^{2t}-1}\right| \wedge (\sigma/(1+x))}{\sqrt{1-e^{-2t}}}\\ = &
\,\left|\frac{x(e^t - 1)}{\sqrt{1-e^{-2t}}} \pm se^t\right|
  \wedge \frac{\sigma}{(1+x)\sqrt{1-e^{-2t}}},    \label{factors2}
\end{align}
where we used  \eqref{lengthj+} and  \eqref{lengthj-}.

\begin{lemma} \label{Fpm}
  For  $\rho \ge 1$, $s>0$ and $\sigma \in (1/2,1)$,
  \begin{equation*}
  \|F_\pm   \|_{L^\infty(0,T]} \lesssim s+1  \qquad \mathrm{and} \qquad  \|F_\pm   \|_{v(\rho),(0,T]} \lesssim s+1,
  \end{equation*}
  where the norm and the seminorm are taken in the variable $t$.
\end{lemma}

\begin{proof}
We have from \eqref{factors2}
\begin{equation}\label{faktor+}
 |F_\pm| \le \frac{xe^t(e^t - 1)}{\sqrt{e^{2t}-1}} +se^t
  \le \frac{4xe^t\,t}{\sqrt{2t}}\, + es \lesssim   x\sqrt{t} + s  \lesssim s+1,
\end{equation}
where the second inequality comes from \eqref{simple} and
the last step uses \eqref{T}. The first inequality of the lemma is verified.

For the second inequality, we will apply Lemma \ref{var}(b).
The factors $F_\pm$ are not always monotone in $t$, but we will split  the interval $(0, T(s,\sigma)]$ into  subintervals where they are monotone.         % as functions of $t$.
The splitting may depend on $s,\,\sigma$ and $x$, but the number of subintervals
will be no larger than $C$.
                                               %must be bounded by an absolute constant.
 This will be done in several steps.

To begin with, we consider only  $F_-$.
We split  $(0, T(s,\sigma)]$ at $t = \widetilde t(s)$ if  $\widetilde t(s) < T(s,\sigma)$, and also at  $t = t_0(s)$ if  $t_0(s) < T(s,\sigma)$.

The splitting then continues, and now we take both  $F_+$ and  $F_-$ into account.  The next split depends on which of the quantities  $\left|x(e^t - 1)\pm s\sqrt{e^{2t}-1}\right|$  and $\sigma/(1+x)$, occurring in the minimum in \eqref{factors1}, is the smaller. Since
 $\left|x(e^t - 1)\pm s\sqrt{e^{2t}-1}\right|$  is monotone in each subinterval obtained so far, this
  may give one split for  $F_+$  and another for  $F_-$.  Next, observe that in the subintervals where
  $\sigma/(1+x)$ is the smaller, we see from  \eqref{factors2} that  $F_\pm = \sigma/\left((1+x)\sqrt{1-e^{-2t}}\,\right)$, which is monotone.
  It only remains to consider the case when  $\left|x(e^t - 1)\pm s\sqrt{e^{2t}-1}\right|$  is the smaller quantity in  \eqref{factors1}. Then  \eqref{factors2} shows that
   $F_\pm = P_\pm$, where
  \begin{equation*}
 P_\pm  =  \frac{x(e^t - 1)}{\sqrt{1-e^{-2t}}} \pm se^t.
  \end{equation*}
The derivative of $P_\pm$ is seen to vanish precisely when
 \begin{equation*}
  x\, \frac{e^t  -2e^{-t} + e^{-2t}}{( 1 -e^{-2t} )^{3/2}}   = \mp  se^t.
  \end{equation*}
In this equation, we multiply by the denominator and square both sides. After multiplication by a suitable power of $e^t$, the result will be a polynomial equation in $e^t$,  with only a bounded number of solutions. Thus we can split our subintervals further, into intervals where $P_\pm$ is monotone.

This ends the splitting, and Lemma \ref{var}(b) implies the second inequality of  Lemma~\ref{Fpm}.
\end{proof}

We can now finish the proof of  Proposition \ref{prp}.
Applying  Lemma  \ref{var_prod} to the two products in \eqref{means}, we get
\begin{align*}
 & \| R_t^{s,\sigma}g(x)\|_{v(\rho), (0, T]}  \\ & \le
  \|F_+   \|_{L^\infty(0,T]} \|M_{|J_t^+(s,\sigma)|}^+ \,g(x)\|_{v(\rho), (0, T]} +
  \|F_+   \|_{v(\rho),(0,T]}\, \|M_{|J_t^+(s,\sigma)|}^+ \,g(x)\|_{L^\infty(0,T]} \\ & +
  \|F_-   \|_{L^\infty(0,T]}\,
  \left( \|M_{|J_t^-(s,\sigma)|}^+ \,g(x)\|_{v(\rho), (0, T]}+  \|M_{|J_t^-(s,\sigma)|}^- \,g(x)\|_{v(\rho), (0, T]}\right)  \\ & +
  \|F_-   \|_{v(\rho),(0,T]}\, \left(\|M_{|J_t^-(s,\sigma)|}^+ \,g(x)\|_{L^\infty(0,T]} + \|M_{|J_t^-(s,\sigma)|}^- \,g(x)\|_{L^\infty(0,T]}\right).
\end{align*}
Using  Lemma \ref{Fpm} together with \eqref{weak+}, \eqref{maxfcn+}, \eqref{weak-}, and \eqref{maxfcn-}, we conclude
that for a.a. $x$
\begin{multline*}
  \| R_t^{s,\sigma}g(x)\|_{v(\rho), (0, T]} \\
  \lesssim \,  \,(1+s)  \,
     \big(
  \|M_\tau^+\, g(x)\|_{v(\rho), \Bbb R_+}
  + \|M_\tau^-\, g(x)\|_{v(\rho), \Bbb R_+}        %_{{v(\rho), }  %\Bbb R_+
          %     \\
 +        \mathcal M^+ g(x) +    \mathcal M^- g(x)
  \big).
  \end{multline*}
 The four terms to the right here are independent of  $s$ and $\sigma$, and we can insert this estimate in \eqref{Ht} and integrate with respect to $s$ and $\sigma.$
The result will be
 \begin{multline*}
 \| \mathcal H_t^{\mathrm{loc}}  g(x)\|_{v(\rho), (0,1]} \lesssim
  \|M_\tau^+\, g(x)\|_{v(\rho), \Bbb R_+}
  + \|M_\tau^-\, g(x)\|_{v(\rho), \Bbb R_+}        %_{{v(\rho), }  %\Bbb R_+
          %     \\
 +        \mathcal M^+ g(x) +    \mathcal M^- g(x)
  \end{multline*}
  for a.a. $x \in \widetilde{I}_j \cap \mathbb R_+$.
  In view of Theorem \ref{jones}, this shows that the operator given by $\| \mathcal H_t^{\mathrm{loc}}  g(x)\|_{v(\rho), (0,1]}$ is of weak type (1,1) as stated in  Proposition \ref{prp}. This ends the proofs of
  Proposition \ref{prp} and also that of Proposition \ref{locsmallt}.

                                  %\end{proof}

  \vskip70pt

 %\newpage

                   %  \section{APPENDIX}

   \noindent APPENDIX.  Asymptotics of  $x_j$

\vskip4pt

\noindent Claim:
\begin{equation*}
  x_j = 2\,\sqrt{j} -1 + O\left(\frac{1}{\sqrt j}\right),  \qquad j \to +\infty.
\end{equation*}

To prove this, let $z_j  = 1 + x_j$ for $j = 1,2,\dots$. The recursion formula  says that
$z_{j+1} - z_j = 1/z_j +1/z_{j+1} $. We have
\begin{equation}\label{squardiff}
  z_{j+1}^2 - z_j^2 = (z_{j+1} + z_j)\left(\frac1{z_j} + \frac1{z_{j+1}}\right) =
2 + \frac{z_{j+1}}{z_{j}} + \frac{z_{j}}{z_{j+1}} > 2.
\end{equation}
             %The recursion formula then implies
             Writing $z_j^2$ as a telescoping sum, we obtain
\begin{equation*}
 z_j^2 = z_0^2 + \sum_{0}^{j-1} \left( z_{\nu+1}^2 - z_\nu^2 \right) \ge 1 +2j,
\end{equation*}
and so
$z_{j+1} - z_j \le 2/\sqrt{1+2j}$. We continue with the  sum in \eqref{squardiff}, getting
\begin{multline*}
  z_{j+1}^2 - z_j^2 = 2 + \frac{z_{j+1}-z_{j}}{z_{j}} + 1 + \frac{z_{j}-z_{j+1}}{z_{j+1}} + 1 =
  4 + (z_{j+1}-z_{j})\left(\frac1{z_j} - \frac1{z_{j+1}}\right)  \\ =
  4 +\frac{(z_{j+1}-z_{j})^2}{z_{j}z_{j+1}} = 4 +O\left(\frac1{(1+2j)^{2}}  \right).
\end{multline*}
Summing as above, we find
\begin{equation*}
 z_j^2 =   4j + O(1),
\end{equation*}
and so $z_j =   2\,\sqrt j + O(1/\sqrt j)$. This proves the claim.

    %One can get more precise expressions by iterating the argument.

\end{document}